\documentclass{amsart}
\pdfoutput=1
\usepackage{amssymb,latexsym,amsmath,amscd,graphicx,graphics,epic,eepic,enumerate, amsfonts, amssymb, mathrsfs, pifont}

\theoremstyle{plain}
\newtheorem{theorem}{Theorem}[section]
\newtheorem{lemma}[theorem]{Lemma}
\newtheorem{proposition}[theorem]{Proposition}
\newtheorem{corollary}[theorem]{Corollary}

\theoremstyle{definition}
\newtheorem{question}[theorem]{Question}
\newtheorem{definition}[theorem]{Definition}

\theoremstyle{remark}
\newtheorem{remark}[theorem]{Remark}

\def \R{\mathbb{R}}
\def \Z{\mathbb{Z}}

\def \S{\Sigma}

\begin{document}

\title{Relative Knot Invariants: Properties and Applications}

\author{Georgi D. Gospodinov}

\address{Franklin W. Olin College of Engineering, 1000 Olin Way, Needham, MA 02492}

\email{georgi.gospodinov@olin.edu}


\begin{abstract}
We state Bennequin inequalities in the relative case, and show that the relative invariants are additive under relative connected sums. We show they exhibit similar limitations as their classical analogues. We study relatively Legendrian simple knots and give some classification results.
\end{abstract}

\maketitle

\section{Introduction}

Classifying Legendrian and transverse knots in contact 3-manifolds has been an important part of the recent development of 3-manifold topology. One of the breakthroughs in this direction came about with the work of GIroux and the theory of convex surfaces (see \cite{aebisher, etnyrehonda:knots1, etnyrehonda:knots2, giroux:convex}). Ideas of convex surface theory are usually applied to null-homologous knots in a contact 3-manifold. Our goal is to apply them in the case when a knot is homologous to another ``reference" knot.

In \cite{georgi}, we defined the following relative invariants.

\begin{definition}
Let $K$ and $J$ be homologous Legendrian knots in a contact 3-manifold $(M,\xi)$ oriented accordingly with $K\cup J=\partial\Sigma$ for an oriented embedded Seifert surface $\Sigma$ so that $[\partial\Sigma]=[K]-[J]$. Define the {\em Thurston-Bennequin invariant of $K$ relative to $J$} by $$\widetilde{tb}_\Sigma(K,J):=tw_K(\xi,Fr_{\Sigma})-tw_{J}(\xi,Fr_{\Sigma}),$$ where $Fr_\Sigma$ denotes the Seifert framing that $K$ (resp. $J$) inherits from $\Sigma$, and $tw(\xi,Fr_\Sigma)$ denotes the number of $2\pi$-twists (with sign) of the contact framing relative to $Fr_\Sigma$ along $K$ or $J$. For push-offs $K'$ and $J'$ of $K$ and $J$ in the direction normal to the contact planes, $\widetilde{tb}_\Sigma(K,J)=K'\cdot\Sigma-J'\cdot\Sigma=lk_\Sigma(K',K)-lk_\Sigma(J',J)$.
\end{definition}

\begin{definition} 
Let $K$ and $J$ be homologous Legendrian knots in a contact 3-manifold $(M,\xi)$ oriented accordingly with $K\cup J=\partial\Sigma$ for an oriented embedded Seifert surface $\Sigma$ so that $[\partial\Sigma]=[K]-[J]$. The restriction to $K$ of the trivialized contact 2-plane field $\xi\rvert_\Sigma$ gives a map $\sigma:\xi\rvert_K\rightarrow K\times\R^2$, under which a non-zero tangent vector field $v_K$ to $K$ traces out a path of vectors in $\R^2$. We can then compute the winding number $w_\sigma(v_K)$ and similarly for $J$. Then define the \textit{relative rotation number of $K$} by $$\widetilde{r}_\Sigma(K,J):=w_\sigma(v_K)-w_\sigma(v_J).$$ Equivalently, $\widetilde{r}_\Sigma(K,J)=e(\xi,v_K\cup v_J)([\Sigma])$.
\end{definition}

\begin{definition} 
Let $K$ and $J$ be homologous transverse knots in a contact 3-manifold $(M,\xi)$ oriented accordingly with $K\cup J=\partial\Sigma$ for an oriented embedded Seifert surface $\Sigma$ so that $[\partial\Sigma]=[K]-[J]$. The contact 2-plane field $\xi$ is trivial over $\Sigma$, so there exists a nonzero vector field $v$ in $\xi\rvert_\Sigma$. Take $K'$ and $J'$ to be the push offs of $K$ and $J$ in the direction of $v$. Then define the {\em relative self-linking number} of $K$ with respect to $J$ by 
$$\widetilde{sl}_\Sigma (K,J):=K'\cdot\Sigma-J'\cdot\Sigma.$$
\end{definition}

In what follow, we establish relative versions of the Bennequin inequalities and develop some prototypical examples. We describe relative connected sums of Legendrian and transverse knots and study the additivity of the relative invariants, following the foundational work of Etnyre-Honda \cite{etnyrehonda:knots2}. We show that the relative invariants exhibit similar limitations as their classical analogues, in particular, the relative Thurston-Bennequin invariant and the relative rotation number are not able to distinguish relative connected sums of the Chekanov knots \cite{chekanov} which are smoothly isotopic, have equal relative invariants, but are not Legendrian isotopic. We study basic knot types which can be classified by their relative invariants, and give a generalization of the structure theorem of Etnyre-Honda \cite{etnyrehonda:knots2} which classifies Legendrian knots in a relative knot type in terms of their relative connected sum prime components.

\section{Acknowledgements} 

I would like to express deep gratitude to John Etnyre for his guidance and help with many fundamental and technical aspects of this work. I am also grateful to my advisor Danny Ruberman for his patience and support throughout my graduate years when the ideas of this paper were developed.

\section{Background}

We briefly recall some facts from contact geometry and convex surface theory. This is far from a complete introduction to the subject, and the reader should consult the more complete treatment in \cite{aebisher, etnyrehonda:knots1, etnyrehonda:knots2, giroux:convex}.

\begin{definition} An (transversely) oriented positive contact structure $\xi$ on $M$ is an oriented 2-plane field $\xi\subset TM$ for which there is a 1-form $\alpha$ such that $\xi=\ker\alpha$ and $\alpha\wedge d\alpha>0$ (recall that $M$ is oriented).
\end{definition}

Two contact structures $\xi_1,\xi_2$ on a 3-manifold $M$ are {\em homotopic} if they are homotopic as 2-plane distributions. They are {\em isotopic} if they are homotopic through contact structures. They are {\em contactomorphic} if there is a diffeomorphism $f: M\rightarrow M$ such that $f$ sends one of the contact structures to the other, i.e., $f_\ast(\xi_1)=\xi_2$. Then $f$ is called a {\em contactomorphism}.

Perturbing a contact structure occurs only through perturbing the ambient manifold, as the theorem below states.

\begin{theorem}[Gray Stability] Given a 1-parameter family of contact structures $\xi_t, t\in[0,1]$, there is a 1-parameter family of diffeomorphisms $f_t:M\rightarrow M$ such that $(f_t)_\ast(\xi_0)=\xi_t$ for all $t$.
\end{theorem}

A smooth oriented embedding of $S^1$ in a contact 3-manifold $(M,\xi)$ is called a \textit{Legendrian knot} if it is everywhere tangent to the contact planes. It is a {\em transverse knot} if it is everywhere transverse to the contact planes.

If $K\subset\Sigma$ is a simple closed Legendrian curve in an embedded surface $\Sigma$ in a contact 3-manifold $(M,\xi)$, then $tw_\Sigma(K)$ is the {\em twisting} of $\xi$ along $K$ relative to the Seifert framing $Fr_\Sigma$. That is, both $\xi$ and $\Sigma$ give $K$ a framing (a trivialization of its normal bundle) by taking a vector field normal to $K$ and tangent to $\xi$ or $\Sigma$, respectively (note that $\xi$ is trivializable over $\Sigma$). Then $tw_\Sigma(K)$ measures the number of $2\pi$-twists (as we traverse the oriented $K$) of the vector field corresponding to $\xi$ relative to the vector field coming from $\Sigma$. By convention, left-handed twists are negative and right-handed twists are positive. Equivalently, take a push-off $K'$ of $K$ along a vector field transverse to $\xi$. Then $tw_\Sigma(K)$ is equal the signed intersection of $K'$ with $\Sigma$, $tw_\Sigma(K)=K'\cdot\Sigma$, or the linking number of $K$ with $K'$.

Let $\Sigma$ be an oriented compact surface embedded in a contact 3-manifold $(M,\xi)$. If $\partial \Sigma$ is nonempty, assume that it is Legendrian. Then the line field $l_p=\xi_p\cap T_p\S$, $p\in\Sigma$, integrates to a singular foliation on $\Sigma$ called the \textit{characteristic foliation}, denoted $\Sigma_\xi$.

The contact structure $\xi$ is called {\em overtwisted} if there is an embedded disc $D$ such that $D_\xi$ contains a closed leaf. Such a disc is called an overtwisted disc. If there are no overtwisted discs in $\xi$, then the contact structure is called {\em tight}.

Now we turn to the theory of convex surfaces, which have been a very useful tool in the study of 3-dimensional contact manifolds.

\begin{definition} Let $\Sigma$ be an oriented compact surface embedded in a contact 3-manifold $(M,\xi)$. If $\partial \Sigma$ is nonempty, assume that it is Legendrian. Then $\Sigma$ is called \textit{convex} if there exists a \textit{contact vector field} $X$ that is transverse to $\Sigma$ (a contact vector field $X$ is a vector field whose flow preserves the contact structure).
\end{definition}

Any closed surface is $C^\infty$-close to a convex surface. If $\Sigma$ has Legendrian boundary with $tw_\Sigma(K)\leq 0$ for all components $K$ of $\partial\Sigma$, then after a $C^0$-small perturbation of $\Sigma$ near the boundary (but fixing the boundary), $\Sigma$ will be $C^\infty$-close to a convex surface.

\begin{definition} Let $\Sigma$ be a convex surface with $X$ a transverse contact vector field. The set $\Gamma_\Sigma=\{p\in \Sigma\ \rvert\ X_p\in\xi_p\}$ is an embedded multi-curve on $\Sigma$ called the {\em dividing} set. 
\end{definition}

\begin{proposition} Let $\mathcal{F}$ be a singular foliation on $\Sigma$ and let $\Gamma$ be a multi-curve on $Sigma$. The multi-curve is said to {\em divide} $\mathcal{F}$ if
\begin{enumerate}[(a)]
\item $\Gamma_\Sigma$ is transverse to $\mathcal{F}$
\item $\Sigma\setminus \Gamma_\Sigma=\Sigma_+\amalg \Sigma_-$
\item there is a vector field $X$ and a volume form $\omega$ on $\Sigma$ such that 
\begin{enumerate}
		\item[(i)] $X$ directs $\mathcal{F}$ (that is, it is tangent to $\mathcal{F}$ at non-singular points and $X=0$ at the singular points of $\mathcal{F}$) 
		\item[(ii)] the flow of $X$ expands $\omega$ on $\Sigma_+$ and contracts $\omega$ on $\Sigma_-$
		\item[(iii)] and $X$ points transversely out of $\Sigma_+$.
		\end{enumerate}
\end{enumerate}
\end{proposition}

\begin{theorem}[Giroux's Criterion] Let $\Sigma$ be a convex surface in a contact 3-manifold $(M,\xi)$. Then $\Sigma$ has a tight neighborhood in $M$ if and only if $\Sigma\neq S^2$ and $\Gamma_\Sigma$ contains no contractible curves or $\Sigma= S^2$ and $\Gamma_\Sigma$ is connected.
\end{theorem}

\section{Generalized Bennequin inequalities}

Let $\Sigma$ be an embedded surface in a contact 3-manifold $(M,\xi)$ with $\partial \Sigma\neq\emptyset$ having multiple components. Let $\mathfrak{F}$ be the singular characteristic foliation on $\Sigma$. Isotop $\Sigma$ ($C^\infty$-small) away from $\partial\Sigma$ so that the singularities of $\mathfrak{F}$ are isolated elliptic and hyperbolic (see \cite{etnyre:intro, giroux:convex}). Let $e_\pm$ be the number of positive/negative elliptic singularities and $h_\pm$ be the number of positive/negative hyperbolic singularities. The Poincar\'e-Hopf theorem says that $\chi(\Sigma)=(e_++e_-)-(h_++h_-)$.

For a transverse knot $K$ with Seifert surface $\Sigma$ in a contact 3-manifold $(M,\xi)$, consider a non-zero section in the trivialization $\xi\rvert_\Sigma$ and take a push-off $K'$ of $K$ along this section. The \textit{self-linking number of $K$} is defined by $sl(K):=K'\cdot\Sigma$. Let $e(\xi)([\Sigma])$ denote the Euler class of $\xi\rvert_\Sigma$. Let $\mathfrak{F}$ be the characteristic foliation which flows transversely out of $K=\partial \Sigma$. Consider the graph $G=\{(x,p)\in\xi\rvert_\Sigma\ \rvert\ p=v(x), x\in\Sigma, p\in \xi_x\}$ of $v$ which directs $\mathfrak{F}$. $G$ is a surface in the 4-manifold $\xi\rvert_\Sigma$, and the zero section is another surface given by $\{(x,0)\in \xi\rvert_\Sigma\ \rvert\ x\in\Sigma\}$, so the Euler class of $\xi\rvert_\Sigma$ is the oriented intersection number of these two surfaces. Counting singularities with signs, we have $-sl(K)=e(\xi)([\Sigma])=(e_+-h_+)-(e_--h_-)$. The Poincar\'e-Hopf theorem yields Eliashberg's equation $\chi(\Sigma)+e(\xi)([\Sigma])=2(e_+-h_+)$. Convex surface theory gives us that $e_+=0$ (see \cite{yasha:legknots, etnyre:intro}) and we obtain the classical Bennequin inequality $sl(K)\leq-\chi(\Sigma)$ (\cite{bennequin}). 

This approach generalizes directly for $\Sigma$ with transverse $\partial \Sigma=K_1\cup\cdots\cup K_m$, we have $-(sl(K_1)+\cdots+sl(K_m))=e(\xi)([\Sigma])=(e_+-h_+)-(e_--h_-)$.

\begin{lemma}\label{generalb}(Generalized Bennequin inequality) 
Given $\Sigma$ with transverse $\partial \Sigma=K_1\cup\cdots\cup K_m$ in a tight contact 3-manifold, $sl(K_1)+\cdots+sl(K_m)\leq -\chi(\Sigma)$.
\end{lemma}

Following Eliashberg \cite{yasha:legknots} and Etnyre \cite{etnyre:intro, etnyre:knots}, consider a Legendrian knot $K$ with Seifert surface $\Sigma$ and an annulus $A=S^1\times [-1,1]$ in a standard neighborhood around $K$ such that $A$ is transverse to $\xi$ and $K$ is the only closed leaf on the characteristic foliation of $A$. Then take the union of $\Sigma$ with the appropriate part of $A$ to form a Seifert surface $\Sigma_\pm$ for the knot $\gamma_\pm=S^1\times\{\pm1\}$. If the neighborhood is chosen so that $\partial A$, the $\Sigma_\pm$ are isotopic to $\Sigma$, and the Euler characteristic of $\chi(\Sigma_\pm)=\chi(\Sigma)$ because the part of $A$ in each Seifert surface does not contribute to $\chi(\Sigma_\pm)$. Then  $sl(K_{1\pm})+\cdots+sl(K_{m\pm})=\Big{(}tw_{K_1}(\xi,Fr_\Sigma)\mp r(K_1)\Big{)}+\cdots+\Big{(}tw_{K_m}(\xi,Fr_\Sigma)\mp r(K_m)\Big{)}$. 

\begin{lemma}\label{generaltb}(Generalized Thurston-Bennequin inequality) 
Given $\Sigma$ with transverse $\partial \Sigma=K_1\cup\cdots\cup K_m$ in a tight contact 3-manifold, we have $$tw_{K_1}(\xi,Fr_\Sigma)+\cdots+ tw_{K_m}(\xi,Fr_\Sigma)+\rvert r(K_1)+\cdots+r_(K_m)\rvert \leq-\chi(\Sigma).$$
\end{lemma}

This observation has several important consequences.

\begin{lemma} 
Let $K$ and $J$ be homologous Legendrian knots in a tight contact 3-manifold $(M,\xi)$, then $\widetilde{tb}_\Sigma(K,J)$ is bounded above.
\end{lemma}

\begin{proof} 
By Lemma \ref{generaltb}, $tw_K(\xi,Fr_\Sigma)+ tw_J(\xi,Fr_\Sigma) +\rvert r(K)+ r(J)\rvert \leq-\chi(\Sigma)$ yields $\widetilde{tb}_\Sigma(K,J)+\rvert r_\Sigma(K)+ r_\Sigma(J)\rvert \leq-\chi(\Sigma)-2tw_J(\xi,Fr_\Sigma)$ or $\widetilde{tb}_\Sigma(K,J)\leq-\chi(\Sigma)-2tw_J(\xi,Fr_\Sigma)$. The quantity $-\chi(\Sigma)-2tw_J(\xi,Fr_\Sigma)$ is fixed because $J$ is fixed. 
\end{proof}

This argument generalizes directly for a knot homologous to multiple knots.

\begin{lemma}\label{reltb} 
If $K,J_1,\dots,J_m$ are  Legendrian with $K\cup J_1\cup J_2\cup \cdots \cup J_m=\partial\Sigma$ in a tight contact 3-manifold $(M,\xi)$, then $\widetilde{tb}_\Sigma(K,J_1\cup \cdots\cup J_m)$ is bounded above.
\end{lemma}

\begin{remark} The above bound depends on the $J_i$ while $\widetilde{tb}_n(K,J)$ in Theorem \ref{tbnbound} is bounded by $Tb(\varphi(K))-n$ for $\varphi(K)\subset(S^3,\xi_{std})$ and even though $J$ is also fixed, instead of using the Seifert surface for $K\cup J$ directly, we find a Seifert surface $\Sigma'$ for $\varphi(K)$ and use the bound on $tb_{\Sigma'}(\varphi(K))$. We want to compare the two approaches. Since $\Sigma'=\Sigma\cup D$, $\chi(\Sigma')=\chi(\Sigma)+\chi(D)$. The two bounds are $\widetilde{tb}_\Sigma(K,J)\leq Tb_{\Sigma'}(\varphi(K))-tw_J(\xi_n,Fr_\Sigma)$ and $\widetilde{tb}_\Sigma(K,J)\leq -\chi(\Sigma)-2tw_J(\xi_n,Fr_\Sigma)$. Since $Tb_{\Sigma'}(\varphi(K))\leq -\chi(\Sigma')$, we have $\widetilde{tb}_\Sigma(K,J)\leq -\chi(\Sigma')-tw_J(\xi_n,Fr_\Sigma)$. Also $tb_D(\varphi(K))+\rvert r_D(\varphi(K))\rvert\leq-\chi(D)$ and $tb_D(K_0)=tw_J(\xi_n,Fr_\Sigma)$, which implies that $tw_J(\xi_n,Fr_\Sigma)\leq -\chi(D)$, which yields $\widetilde{tb}_\Sigma(K,J)\leq-\chi(\Sigma')-tw_J(\xi_n,Fr_\Sigma)$. Both bounds are smaller than $-\chi(S)-tw_J(\xi_n,Fr_\Sigma)$, but we do not have a direct way of comparing them by just using classical methods. This relates to the problem of the exactness of the Thurston-Bennequin inequality.
\end{remark}

\section{Additivity of the relative invariants}

We study the additivity of the relative invariants under versions of connected sum. The results build up on the work of Etnyre-Honda \cite{etnyrehonda:knots2}. Recall the following.

\begin{theorem}\label{colin}(Colin \cite{colin}) 
Denote by $Tight(M)$ the space of tight contact 2-plane fields on a 3-manifold M. Then given contact 3-manifolds $M_1, M_2$, there is an isomorphism $\pi_0(Tight(M_1)\times \pi_0(Tight(M_2))\xrightarrow{\cong}\pi_0(Tight(M_1\#M_2)).$
\end{theorem}

\begin{remark} \label{contactsum}(Contact connected sum \cite{etnyrehonda:knots2, torisu}) 
Let $(M_i,\xi_i),\ i=1,2,$ be tight contact 3-manifolds. Choose points $p_i\in M_i$ and a \textit{standard contact 3-ball $B_i$} around each $p_i$ (by Darboux's theorem, $(B_i,\xi_i\rvert_{B_i})$ is contactomorphic to a 3-ball around the origin in $(S^3,\xi_{std})$). Note $\partial B_i$ is $C^\infty$-close to a convex 2-sphere with a single dividing curve (Giroux's Criterion, \cite{giroux:criterion}), and Giroux's Flexibility Theorem \cite{giroux:convex} allows us to arrange that $\partial B_i$ have diffeomorphic foliations so the $B_i$ are contactomorphic (\cite{yasha:20}), and there is an orientation-reversing diffeomorphism $f:\partial(M_1\setminus B_1)\rightarrow \partial(M_2\setminus B_2)$ that maps the characteristic foliation on $\partial(M_1\setminus B_1)$ to the characteristic foliation on $\partial(M_2\setminus B_2)$. Then the \textit{contact connected sum} $(M_1,\xi_1)\#(M_2,\xi_2)=\big{(}(M_1\setminus B_1)\cup_f (M_2\setminus B_2),\xi_1\#_f\xi_2\big{)}$ yields a tight contact 3-manifold and is independent of the choice of $B_i,p_i$, and $f$. Moreover, every tight contact structure on $M$ arises as the contact connected sum of a unique pair $(\xi_1,\xi_2)$.
\end{remark}

\begin{remark}\label{legsum} (Legendrian connected sum \cite{etnyrehonda:knots2}) 
The Legendrian connected sum is a relative version of the contact connected sum. In $(S^3,\xi_{std})$, it can easily be described using the front projection of two Legendrian knots $K_1$ and $K_2$ as joining a right cusp of $K_1$ and a left cusp of $K_2$ (well-defined by the uniqueness of the front projection). By Theorem \ref{colin}, the contact structure on $S^3=(S^3,\xi_{std})\#(S^3,\xi_{std})$ is tight so it is isotopic to $\xi_{std}$. In the general construction (\cite{etnyrehonda:knots2}), pick points $p_i\in K_i\subset M_i$ and neighborhoods $B_i$ of the $p_i$. Then use an orientation-reversing diffeomorphism $f:\partial B_1\rightarrow \partial B_2$ to construct the contact connected sum $M_1\#M_2=(M_1\setminus B_1)\cup_f(M_2\setminus B_2)$. This diffeomorphism (Remark \ref{contactsum}) performs exactly what we observed in the front projection, with the cusps at the points $p_i$.
\end{remark}

\begin{lemma}\label{sum} 
In the connected sum of $K_1$, $K_2$, $tb(K_1\#K_2)=tb(K_1)+tb(K_2)+1$.
\end{lemma}

\begin{proof} 
This was proved in \cite{etnyrehonda:knots2, torisu}, here we show an argument due to Etnyre (in a personal note) which keeps track of the Seifert surfaces. For Legendrian knots $K_1$ and $K_2$ with $K_i=\partial \Sigma_i,i=1,2$, pick a small arc $a_i$ on $K_i$ and isotop (the interior of) $\Sigma_i$ so that there is a positive elliptic singularity on $a_i$ and no other singularities in a small disc $D_i\subset \Sigma_i$ about $a_i$. Near $\Sigma_i$ but disjoint from it, pick a disc $D_i'$ with boundary $\partial D_i'=a'_i\cup b_i$ where the arc $a'_i$ has a negative elliptic point, the arc $b_i$ is transverse to $\xi_i$, and there are no other singularities in $D_i'$. Now take a Legendrian arc $c_i$ connecting the elliptic points on $a_i$ and $a_i'$. In $(\R^3,\xi_0)$, take a right cusp in the $xz$-plane centered on the $x$-axis lying to the left of the $z$-axis, a left cusp to the right of the $z$-axis, and a Legendrian arc on the $x$-axis connecting the cusps. There is a contactomorphism of a neighborhood of $D_i\cup c_i\cup D_i'$ to two discs in $(\R^3,\xi_0)$ having the cusps in the boundary and the arc on the $y$-axis. Now apply the Legendrian connected sum in the front projection to these cusps along the arc on the $x$-axis. In particular, one can see that the singularity $a_i$ in the characteristic foliation of $\Sigma_i$ before we perform the connected sum contributes a left-handed half-twist to the twisting of the contact planes along $K_i$ relative to the framing from $\Sigma_i$. After the connect sum operation, both these singularities are gone, but all other singularities remain. So there is a net ''+1" to the contact plane twisting along the knot relative to the Seifert framing (here, the Seifert surface is given by $(\Sigma_1\setminus D_1)\cup(\Sigma_2\setminus D_2)$).
\end{proof}

\begin{lemma}\label{split} (Splitting a Legendrian connected sum) 
Consider a Legendrian knot $K=\partial\Sigma$ of knot type $\kappa_1\#\kappa_2$ in a tight contact 3-manifold $(M,\xi)=(M_1,\xi_1)\#(M_2,\xi_2)$. Then $K$ can be split into knots $K_i'\subset(M_i,\xi_i)$ of knot type $\kappa_i\subset(M_i,\xi_i)$ such that $tb(K)=tb(K_1')+tb(K_2')+1$.
\end{lemma}

\begin{proof} 
We modify the Etnyre-Honda construction in \cite{etnyrehonda:knots2} to keep track of the Seifert surfaces. Let $K=\partial\Sigma$ be a Legendrian knot of knot type $\kappa_1\#\kappa_2$ in a tight contact 3-manifold $(M,\xi)=(M_1,\xi_1)\#(M_2,\xi_2)$. There exists a splitting 2-sphere $S$ for $\Sigma$ such that $S\cap \Sigma=\alpha$, an arc with $\partial\alpha=\{x_1,x_2\}\subset K$. Isotop $\Sigma$ so that $x_1$ is a negative singularity on $K$ and $x_2$ is a positive singularity on $K$ in the characteristic foliation of $\Sigma$ (isotop $S$ to make it convex). Note $\alpha$ intersects the dividing set $\Gamma_S$ in $\{x_1,x_2\}$. Take a closed curve $\gamma\subset S$ containing $\alpha$ and isotop $S$ to Legendrian realize $\gamma$  (see \cite{honda:one}). Then $\alpha$ is a Legendrian arc on $\Sigma$ which still intersects $K=\partial \Sigma$ in $\{x_1,x_2\}$. The interior of $\alpha$ contains an odd number of intersections with $\Gamma_S$ (intuitively, it contains an odd number of half-twists of the contact planes relative to $Fr_\Sigma$). Moreover, $\Gamma_S$ consists of a single closed curve, so the arc $\gamma\subset\Gamma_S$, $\partial\gamma=\{x_1,x_2\}$, which intersects $\alpha$ is ``parallel" to $\alpha$, that is, $\gamma$ and $\alpha$ co-bound a collection of (an even number of) 2-discs on $S$. Consider another arc $\alpha'\subset S$ which is parallel to $\alpha$, is tangent to $\alpha$ at $x_1$ and $x_2$, and which contains a single intersection with $\Gamma_S$ in its interior. Isotop $S$ to Legendrian realize $\alpha'$ and isotop the interior of $\Sigma$ so that $S\cap \Sigma=\alpha'$. Note that $\alpha'$ contains a single left twist of the contact planes with respect to $Fr_S$ and thus - relative to $Fr_\Sigma$. Use $\alpha'$ to complete each of the components of $K$ corresponding to the knot type of $K_1$ or $K_2$, respectively. Thus, we obtain two knots $K_i'\subset(M_i,\xi_i)$ of knot type $\kappa_i\subset(M_i,\xi_i)$ with $tb(K)=tb(K_1')+tb(K_2')+1$. If $\Sigma$ has any other boundary components, the equality $tb(K)=tb(K_1')+tb(K_2')+1$ holds in one of its relative versions (see below).
\end{proof}

First we look at a ``semi-relative'' case when one of the connected summands is homologous to another knot. Note $\widetilde{tb}_{\Sigma}(K_1\#K_2,J)$ is well-defined (see \cite{georgi}). 

\begin{proposition} 
Let $K_1, J\subset(M_1,\xi_1)$ be homologous Legendrian knots and $K_2\subset(M_2,\xi_2)$ be a null-homologous Legendrian knot. Assume the $\xi_i$ are tight, $\widetilde{tb}_{\Sigma_1}(K_1,J)=\widetilde{Tb}(K_1,J)$, and $tb(K_2)=Tb(K_2)$. Then $\widetilde{Tb}(K_1\#K_2,J)=\widetilde{Tb}(K_1,J)+Tb(K_2)+1$, where $J$ in $\widetilde{Tb}(K_1\#K_2,J)$ is the image of $J\subset M_1$ under the connected sum.
\end{proposition}

\begin{proof} 
Because $\xi_i$ are tight, $\widetilde{tb}(K_1,J)$ and $tb(K_2)$ are bounded. For Legendrian knots $K_1, K_2, J$ as above, Lemma \ref{sum} gives $\widetilde{Tb}(K_1,J)+Tb(K_2)+1=\widetilde{tb}(K_1\#K_2,J)\leq\widetilde{Tb}(K_1\#K_2,J)$. Conversely, if $\widetilde{tb}(K_1\#K_2,J)=\widetilde{Tb}(K_1\#K_2,J)$, then Lemma \ref{split} implies that $\widetilde{Tb}(K_1\#K_2,J)=\widetilde{tb}(K_1,J)+tb(K_2)+1\leq\widetilde{Tb}(K_1,J)+Tb(K_2)+1$.
\end{proof}

The above also follows from the relative structure theorem (Theorem \ref{relstr}) and directly extends to the case when both summands are homologous to another knot.

\begin{proposition} 
For homologous Legendrian knots $K_i, J_i\subset(M_i,\xi_i)$ with $K_i\cup J_i=\partial S_i,i=1,2$. Assume $\xi_i$ is tight, and $\tilde{tb}(K_1\#K_2,J_1\cup J_2)=\widetilde{Tb}(K_1\#K_2,J_1\cup J_2)$. Then $\widetilde{Tb}(K_1\#K_2,J_1\cup J_2)=\widetilde{Tb}(K_1,J)+\widetilde{Tb}(K_2,J_2)+1$, where $J_1\cup J_2$ in the term $\widetilde{Tb}(K_1\#K_2,J_1\cup J_2)$ is a knot in $M_1\#M_2$ .
\end{proposition}

\begin{remark}\label{relsum} (Relative Legendrian connected sum) 
 Consider homologous Legendrian knots  $K_i, J_i\subset(M_i,\xi_i)$ with Seifert surface $\Sigma_i, i=1,2$. Take an arc $\alpha_i\subset \Sigma_i$ with $\partial \alpha_i\subset \partial \Sigma_i$ such that $\alpha_i$ runs from $K_i$ to $J_i$. 
 
Take a neighborhood $B_i\subset M_i$ of $\alpha_i$ (with convex boundary) and an orientation-reversing diffeomorphism $f:\partial B_1\rightarrow \partial B_2$. Form the connected sum $M_1\#_{\alpha_1,\alpha_2} M_2\:= (M_1\setminus B_1)\cup_f(M_2\setminus B_2)$ as follows. First, $B_i\cap \Sigma_i$ is a 2-disc $D_i\subset \Sigma_i$ with four corners, $a_i,b_i,c_i,d_i$, whose boundary $\partial D_i$ is a union of four arcs. Isotop the interior of $\Sigma$ so that the four corners of $\partial D_i$ are singularities in the foliation of $\Sigma$.

We have $\partial D_i=\gamma_{K_i}\cup\gamma_{J_i}\cup\gamma_i'\cup\gamma_i''$ where $\gamma_{K_i}\subset K_i$ with $\partial\gamma_{K_i}=\{a_i,b_i\}$, $\gamma_{J_i}\subset J_i$ with $\partial \gamma_{J_i}=\{c_i,d_i\}$ (Figure 14). Isotop $\partial B_i$ to Legendrian realize $\partial D_i$ and isotop the interior of $D_i$ to make it convex. By tightness, $tb(D_i)\leq-1$. Also, $\gamma_{K_i}$ (resp., $\gamma_{J_i}$) intersects $\Gamma_{D_i}$ once and contains a negative half-twist along $K_i$ (resp., $J_i$).

Now take an orientation-reversing diffeomorphism $f:(M_1\setminus B_1)\rightarrow(M_2\setminus B_2)$ such that $f(\partial D_1)=\partial D_2$ and $f(a_1)=a_2,\ f(b_1)=b_2,\ f(c_1)=c_2,\ f(d_1)=d_2$, so that $f(\gamma_{K_1})$ is isotopic to $\gamma_{K_2}$, $f(\gamma_{J_1})$ is isotopic to $\gamma_{K_2}$, $f(\gamma_1')$ is isotopic to $\gamma_2'$, and $f(\gamma_1'')$ is isotopic to $\gamma_2''$ (rel boundary as unoriented arcs). Then in $M_1\#_{\alpha_1,\alpha_2}M_2$, the knots $K=(K_1\setminus \gamma_{K_1})\cup_f(K_2\setminus \gamma_{K_2})$ and $J=(J_1\setminus \gamma_{J_1})\cup_f(J_2\setminus \gamma_{J_2})$ co-bound a surface $\Sigma=(\Sigma_1\setminus D_1)\cup_f(\Sigma_2\setminus D_2)$ and are of type $K_1\#_{\alpha_1,\alpha_2}K_2$ and $J_1\#_{\alpha_1,\alpha_2}J_2$, respectively. Moreover, by construction $tw_K(\xi,Fr_\Sigma)=tw_{K_1}(\xi_1,Fr_{\Sigma_1})+tw_{K_2}(\xi_2,Fr_{\Sigma_2})+1$ and $tw_J(\xi,Fr_\Sigma)=tw_{J_1}(\xi_1,Fr_{\Sigma_1})+tw_{J_2}(\xi_2,Fr_{\Sigma_2})+1$.
\end{remark}

\begin{figure}[htbp]
\centering
\includegraphics[scale=.25]{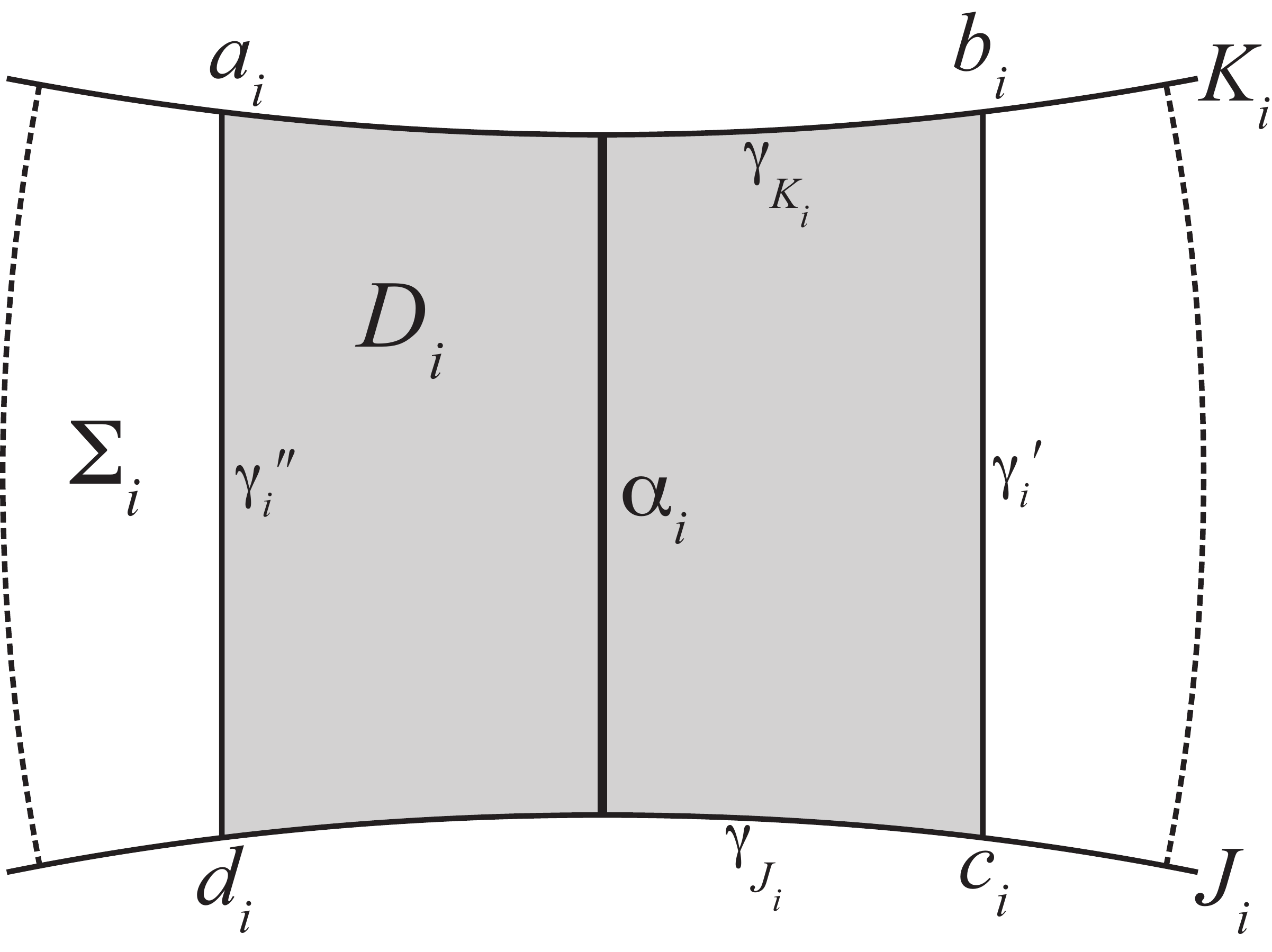}
\caption{Local setup for a relative Legendrian connected sum.}
\end{figure}

The diffeomorphism type of $\Sigma_{\alpha_1,\alpha_2}$ with $\partial\Sigma_{\alpha_1,\alpha_2}=(K_1\#_{\alpha_1,\alpha_2}K_2)\cup (J_1\#_{\alpha_1,\alpha_2}J_2)$ and the link type of $K_1\#_{\alpha_1,\alpha_2}K_2\cup J_1\#_{\alpha_1,\alpha_2}J_2$ depend on the choice of the arcs $\alpha_i$ so this construction is not well-defined as a connected sum of the links $K_i\cup J_i$.
 
\begin{lemma} 
Consider homologous Legendrian knots $K_i$, $J_i$ in a tight contact 3-manifold $(M_i,\xi_i), i=1,2$. In $M_1\#_{\alpha_1,\alpha_2}M_2\:=(M_1\setminus B_1)\cup_f(M_2\setminus B_2)$, the diffeomorphism type of $M_1\#_{\alpha_1,\alpha_2}M_2$, the isotopy type of $\xi_1\#_{\alpha_1,\alpha_2}\xi_2$, and the knot type of $K_1\#_{\alpha_1,\alpha_2}K_2$ and $J_1\#_{\alpha_1,\alpha_2}J_2$ are independent of the choices of the $\alpha_i, B_i,f$.
\end{lemma}

\begin{proof} 
Consider the relative Legendrian connected sums of $(K_1,J_1,M_1,\xi_1)$ with $(K_2,J_2, M_2,\xi_2)$ along two sets of arcs $\{\alpha_1,\alpha_2\}$ and $\{\beta_1,\beta_2\}$. Then $M_1\#_{\alpha_1,\alpha_2}M_2$ and $M_1\#_{\beta_1,\beta_2}M_2$ are diffeomorphic as smooth manifolds because any two 3-balls in a connected 3-manifold are isotopic, which extends to a global isotopy between the diffeomorphisms $f_\alpha$ and $f_\beta$. Moreover, Colin's theorem gives that $\xi_1\#_{\alpha_1,\alpha_2}\xi_2$ and $\xi_1\#_{\beta_1,\beta_2}\xi_2$ are isotopic. The knots $K_1\#_{\alpha_1,\alpha_2}K_2$ and $K_1\#_{\beta_1,\beta_2}K_2$ are of the same knot type in $M_1\# M_2$ ($\gamma_{K_i}$ do depend on the choice of $\alpha_i$ or $\beta_i$ but we can isotop the $B_i$ (slide them along $K_i$ and $J_i$) so that $\gamma_{K_i}$ coincide for either $\alpha_i$ or $\beta_i$). Similarly,  $J_1\#_{\alpha_1,\alpha_2}J_2$ and $J_1\#_{\beta_1,\beta_2}J_2$ are of the same knot type in $M_1\# M_2$.
\end{proof}

\begin{lemma} 
In the relative Legendrian connected sum of $(K_1,J_1,M_1,\xi_1)$ and $(K_2,J_2,M_2,\xi_2)$, $tw_{K_1\#K_2}(\xi_1\#\xi_2,Fr_{\Sigma_1\#\Sigma_2})$ and $tw_{J_1\#J_2}(\xi_1\#\xi_2,Fr_{\Sigma_1\#\Sigma_2})$ are well-defined and independent of the choice of arcs $\alpha_i$.
\end{lemma}

\begin{proof} 
Consider the relative Legendrian connected sum of $(K_1,J_1,M_1,\xi_1)$ and $(K_2,J_2,M_2,\xi_2)$ along two sets of arcs $\alpha_1,\alpha_2$ and $\beta_1,\beta_2$, where $\alpha_i,\beta_i\subset\Sigma_i$. Then we can isotop $\alpha_i$ so that $\partial \alpha_i$ coincides with $\partial \beta_i$, thus the arcs $\gamma_{K_i}$ and $\gamma_{J_i}$ are the same for both $\alpha_i$ and $\beta_i$. Therefore, under the diffeomorphism between $M_1\#_{\alpha_i,\alpha_2}M_2$ and $M_1\#_{\beta_i,\beta_2}M_2$, $K_1\#_{\alpha_1,\alpha_2}K_2$ is sent to $K_1\#_{\beta_1,\beta_2}K_2$ and $J_1\#_{\alpha_1,\alpha_2}J_2$ is sent to $J_1\#_{\beta_1,\beta_2}J_2$. This diffeomorphism extends to a framing-preserving contactomorphism from a neighborhood of $K_1\#_{\alpha_1,\alpha_2}K_2$ to a neighborhood of $K_1\#_{\beta_1,\beta_2}K_2$ (similarly for $J_1\#_{\alpha_1,\alpha_2}J_2$ and $J_1\#_{\beta_1,\beta_2}J_2$). Under this contactomorphism, the contact framings and the Seifert framings are identified, and the result follows.
\end{proof}

\begin{corollary} 
The value of $\widetilde{tb}_{\Sigma_1\#\Sigma_2}(K_1\#K_2,J_1\#J_2)$ in the relative Legendrian connected sum is independent of the choice of arcs $\alpha_i, B_i, f$.
\end{corollary}

\begin{lemma} \label{reladd}
$\widetilde{tb}(K_1\#K_2,J_1\#J_2)=\widetilde{tb}_{\Sigma_1}(K_1,J_1)+\widetilde{tb}_{\Sigma_2}(K_2,J_2)$.
\end{lemma}

\begin{proof} 
By definition, $\widetilde{tb}(K_1\#K_2,J_1\#J_2)=tw_{K_1\#K_2}(\xi,Fr_\Sigma)-tw_{J_1\#J_2}(\xi,Fr_\Sigma)$ or $\big{(}tw_{K_1}(\xi_1,Fr_{\Sigma_1})+tw_{K_2}(\xi_2,Fr_{\Sigma_2})+1\big{)}-\big{(}tw_{J_1}(\xi_1,Fr_{\Sigma_1})+tw_{J_2}(\xi_2,Fr_{\Sigma_2})+1\big{)}$ and rearranging terms, we obtain $\widetilde{tb}_{\Sigma_1}(K_1,J_1)+\widetilde{tb}_{\Sigma_2}(K_2,J_2)$. 
\end{proof}

\begin{proposition} 
Consider a homologous Legendrian knot pair $(K_i, J_i)$ in a tight contact 3-manifold $(M_i,\xi_i)$ such that $\widetilde{tb}(K_i,J_i)=\widetilde{Tb}(K_i,J_i),i=1,2$. Then in the relative Legendrian connected sum $\widetilde{Tb}(K_1\#K_2,J_1\#J_2)\geq\widetilde{Tb}(K_1,J_2)+\widetilde{Tb}(K_2,J_2)$.
\end{proposition}

\begin{proof} 
By the construction of the relative Legendrian connected sum (Remark \ref{sum}) and by Lemma \ref{reladd} above, $\widetilde{tb}(K_1\#K_2,J_1\#J_2)=\widetilde{Tb}(K_1,J_1)+\widetilde{Tb}(K_2,J_2)$, which implies that $\widetilde{Tb}(K_1\#K_2,J_1\#J_2)\geq\widetilde{Tb}(K_1,J_2)+\widetilde{Tb}(K_2,J_2)$. 
\end{proof}

\begin{remark} (Splitting a relative Legendrian connected sum) Consider a tight contact 3-manifold $(M,\xi)$ with a Legendrian knot pair $(K_1\#K_2,J_1\#J_2)$ which is a relative Legendrian connected sum of $(K_1,J_1,M_1,\xi_1)$ with $(K_2,J_2,M_2,\xi_2)$. Then there exists an embedded splitting 2-sphere $S\subset(M,\xi)$ such that $S\cap\Sigma$ is the union of arcs $\alpha_1$ and $\alpha_2$. Take a 2-disc $D\subset S$ with $\partial D=\gamma'\cup\alpha_1\cup\gamma''\cup\alpha_2$ (Figure \ref{splitting}).

\begin{figure}[htbp]
\centering
\includegraphics[scale=.25]{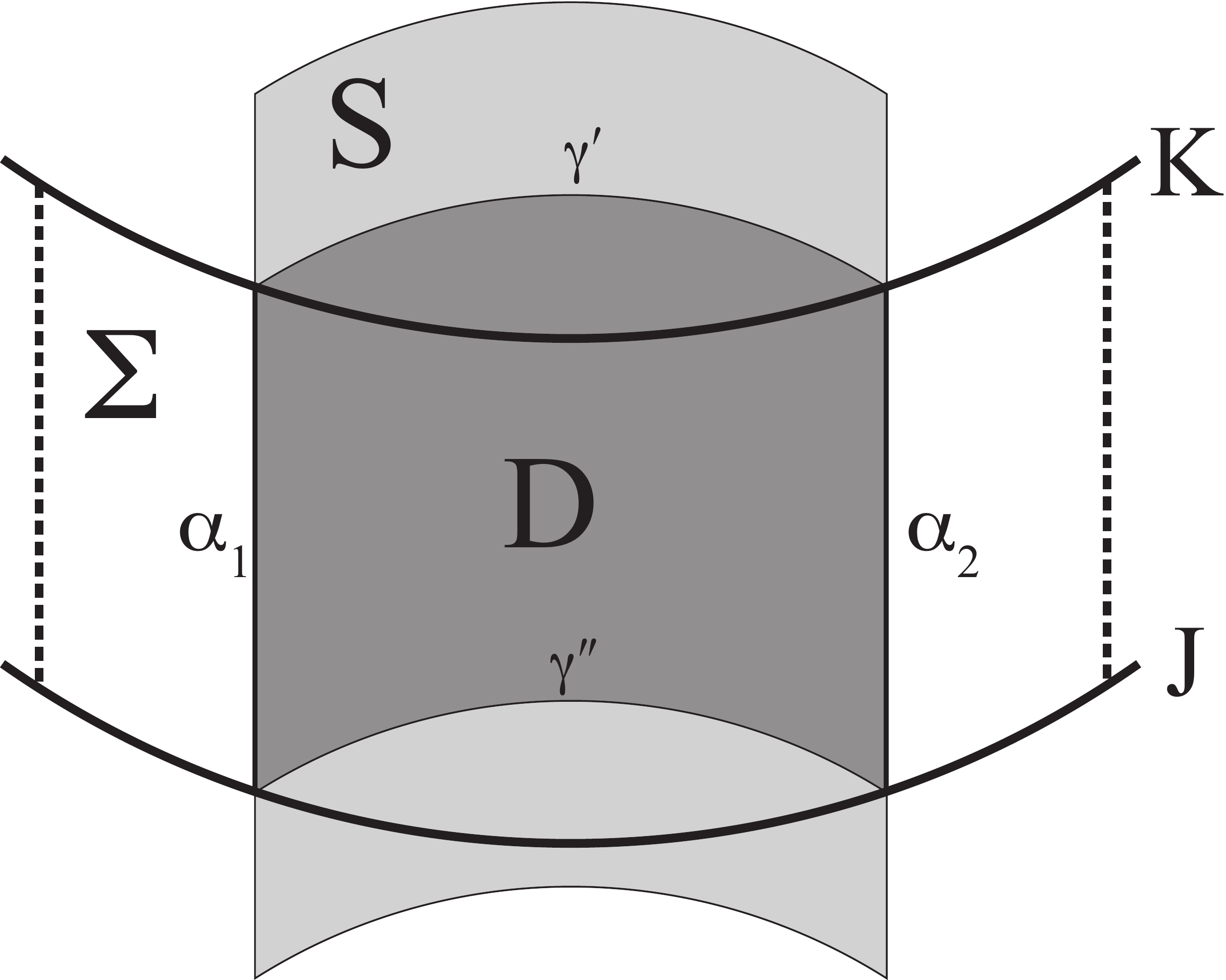}
\caption{Splitting a relative Legendrian connected sum.}\label{splitting}
\end{figure}

Let $M=M_1'\cup M_2'$ with $\partial M_i'=S$, where $S$ is a 2-sphere with appropriate orientation, and complete each $M_i'$ by a standard contact 3-ball $B$. Then $M_i'\cup B$ is diffeomorphic to the original manifold $M_i$ used in the relative Legendrian connected sum. Completing $K$ by $\gamma'$ and $J$ by $\gamma''$ yields knots of type $K_i, J_i\subset M_i$. On the contact level, $\xi\rvert_{M_i'}$ extends uniquely over $B$, and by Colin's theorem, each $M_i'$ has a tight contact structure $\xi_i$. The arcs $\gamma'$ and $\gamma''$ are Legendrian and contain one left-handed half-twist of the contact planes relative to $Fr_S$ each. So the resulting knots are Legendrian and satisfy $\widetilde{tb}(K_1,J_2)+\widetilde{tb}(K_2,J_2)=\widetilde{tb}(K_1\#K_2,J_1\#J_2)$. The diffeomorphism type of the components that $\Sigma$ is split into depends on the arcs $\alpha_i$ and surface $S$. However, knot types of the knots, the diffeomorphism type of the manifolds $M_i$, and the isotopy type of $\xi_i,i=1,2$ are all independent of $\alpha_i$ and $S$.
\end{remark}

\begin{proposition} 
Let $(M,\xi)$be tight and let $(K_1\#K_2,J_1\#J_2)$ be a relative Legendrian connected sum pair with $\widetilde{tb}(K_1\#K_2,J_1\#J_2)=\widetilde{Tb}(K_1\#K_2,J_1\#J_2)$. Then $\widetilde{Tb}(K_1\#K_2,J_1\#J_2)\leq\widetilde{Tb}(K_1,J_2)+\widetilde{Tb}(K_2,J_2)$.
\end{proposition}

\begin{theorem} 
In the relative Legendrian connected sum of homologous Legendrian knot pairs $(K_i,J_i)$ in tight contact 3-manifolds, $\widetilde{Tb}(K_1\#K_2,J_1\#J_2)=\widetilde{Tb}(K_1,J_2)+\widetilde{Tb}(K_2,J_2)$.
\end{theorem}

\begin{remark}\label{addrot} (Additivity of the relative rotation number) 
in some cases, from a global trivialization of the contact structure. The rotation number is measured as the difference of the upward and downward cusps,  so $r(K_1\#K_2)=r(K_1)+r(K_2)$ in $(S^3,\xi_{std})$. By construction, the relative rotation numbers below are well-defined and independent of the choices made in the corresponding connected sum. In the relative Legendrian connected sum, the arcs $\alpha_1$ and $\alpha_2$ contribute the same amount to the rotation number. Since they get sent to one upward and one downward cusp under the (local) contactomorphism to $(S^3,\xi_{std})$, their total contribution to the rotation number is 0. Alternatively, note that the orientation-reversing diffeomorphism used in forming all of the connected sums identifies arcs along the summand knots with the same (and opposite in sign) contributions to the respective rotation numbers.
\end{remark}

\begin{proposition} 
Given homologous Legendrian knots $K_1,J\subset(M_1,\xi_1)$ with $K_1\cup J=\partial S_1$ and a Legendrian knot $K_2\subset(M_2,\xi_2)$ with a Seifert surface $K_2=\partial\Sigma_2$, assume that the $\xi_i$ are tight. Then $\widetilde{r}_{\Sigma}(K_1\#K_2,J)=\widetilde{r}_{\Sigma_1}(K_1,J)+r_{\Sigma_2}(K_2)$, where $J$ in the term $\widetilde{r}_{\Sigma}(K_1\#K_2,J)$ is a knot in $M_1\#M_2$.
\end{proposition}

\begin{theorem} 
For a homologous Legendrian knot pair $(K_i,J_i)$ in a tight contact 3-manifold $(M_i,\xi_i)$ with Seifert surface $\Sigma_i, i=1,2$, assume $(M_i,\xi_i)$ is tight. Then $\widetilde{r}_{\Sigma}(K_1\#K_2,J_1\cup J_2)=\widetilde{r}_{\Sigma_1}(K_1,J_1)+\widetilde{r}_{\Sigma_2}(K_2,J_2)$, where $J_1\cup J_2$ in the term $\widetilde{r}_{\Sigma}(K_1\#K_2,J_1\cup J_2)$ is a knot in $M_1\#M_2$.
\end{theorem}

\begin{remark} Combined with the classical additivity of the self-linking number under transverse connected sums, the constructions in this section (in particular, Remarks \ref{contactsum}, \ref{legsum}, and \ref{relsum}, and Lemmas \ref{sum}, \ref{split}) carry over to the transverse category to give the construction of relative transverse connected sums and additivity of the relative self-linking number under those. In particular, the generalized Bennequin inequality for a tight contact 3-manifold, implies that we would get an additivity of the maximal self-linking numbers of transverse knot pairs under relative transverse connect sums.
\end{remark}

\section{Relative Thurston-Bennequin invariant in $(S^1\times D^2,\xi_n)$}

Consider $S^1\times D^2$ with $\xi_n=\ker\big{(}\sin(2\pi nz)dx+\cos(2\pi nz)dy\big{)},n\geq 1$, in local coordinates $\{z,\{x,y\}\}$, where $\xi_n$ is generated by $\{ \partial/\partial z,\ \cos(2\pi nz)\partial/\partial x-\sin(2\pi nz) \partial/\partial y\}$. The core curve $J=S^1\times\{(0,0)\}$ is Legendrian and, in fact, all curves of the form $\{(z,(x,y))\ \rvert\ x^2+y^2<1,\ z\in S^1\}$ are Legendrian. Intuitively, $S^1\times D^2$ is foliated by Legendrian curves parallel to the core $J$. The contact structure is invariant under translation in the plane $\{z=const.\}$, however, it is not vertically invariant, in particular, it makes $n$ left-handed $2\pi$-twists along each (oriented) Legendrian curve parallel to the core $J$.

Since $S^1\times D^2\setminus J$ retracts to $\partial(S^1\times D^2)\cong S^1\times S^1$, $H_1(S^1\times D^2\setminus J)\cong\Z\oplus \Z$ is generated by oriented $\{\mu,\lambda \}$, where $\lambda=S^1\times\{p\},\ p\in\partial D^2$, and $\mu\subset (S^1\times D^2\setminus J)$ is the boundary of $\{0\}\times D^2$ with $\lambda\cdot\mu=1$. Then for $K\subset (S^1\times D^2\setminus J)$, $[K]=n\mu+m\lambda\in H_1(S^1\times D^2\setminus J),\ n, m\in\Z$. Call $n$ in $[K]=n\mu+m\lambda$ \textit{the linking number $lk(K,J)$ of $K$with $J$}. Alternatively, this is the geometric intersection of $K$ with the annulus $A=S^1\times[(0,0),(1,0)]$ so $lk(K,J)=K\cdot A$.

Consider a Legendrian knot $K$ homologous to $J$ in $(S^1\times D^2,\xi_n)$, $K\cup J=\partial\Sigma$. Denote the \textit{Thurston-Bennequin invariant of $K$ relative to $J$} by $\widetilde{tb}_{n,\Sigma}(K,J):=tw_K(\xi,Fr_{\Sigma})-tw_{J}(\xi,Fr_{\Sigma})$. It is well-defined (with $H_2(S^1\times D^2)=0$ implying it is independent of the Seifert surface $\Sigma$) and depends on the integer  $n$. So for a knot $K$ homologous to the core $J$ and $\partial \Sigma=K\cup J$ in $(S^1\times D^2,\xi_n)$, we will omit the subscript $\Sigma$, and use the notation $\widetilde{tb}_n(K,J)=tw_K(\xi,Fr_{\Sigma})-tw_{J}(\xi,Fr_{\Sigma})$. We want to study Legendrian isotopies of $K$ across $J$.

\begin{lemma} 
Fix a number $r\in(0,1)$ and let $[p,q$ denote the line segment in $D^2\subset \R^2$ from point $p$ to point $q$. There exists an annulus $A=S^1\times [(0,0), (a,b)]$ with $a^2+b^2<r$, such that $K\subset S^1\times\{(x,y): x^2+y^2\leq a^2+b^2\}$ and $K\pitchfork A$.
\end{lemma}

\begin{proof} 
Since $K$ is properly embedded in $S^1\times D^2$, it is contained in a solid torus of the type $S^1\times \{(x,y)\rvert x^2+y^2\leq r'\}$ for some $r'\in(0,1),r'<r$. Parametrize $K$ by $t\longmapsto \{z(t),x(t),y(t)\}$, where $t\in S^1$ and consider the map $f:K\longrightarrow S^1$ given by $f:\{z(t),x(t),y(t)\}\longmapsto \theta\in S^1$ such that $x=\cos{\theta}$ and $y=\sin{\theta}$, in other words, $\theta$ is the angle that the segment $[(0,0),(x(t),y(t))]$ makes with the $x$-axis. So $f:S^1\rightarrow S^1$ is a smooth map and by Sard's theorem, almost every value of the map is a regular value, that is, the differential of $f$ is onto everywhere. Then $f^{-1}(\theta)$ for a given regular value $\theta$ produces a set of transverse intersection points of $K$ with the annulus $A=S^1\times\{(x,y):x^2+y^2\leq r\}$. 
\end{proof}

\begin{lemma} 
The annulus $A$ traces out a Legendrian isotopy from $J=S^1\times\{(0,0)\}$ to $J'=S^1\times\{(c,d)\}$, $c^2+d^2=r$, with $lk(J',K)=0$ and $A\pitchfork K$. It extends to an ambient contact isotopy of $(S^1\times D^2,\xi_n)$ fixing the boundary. 
\end{lemma}

\begin{proof} 
The annulus $A$ is foliated by Legendrian knots of type $S^1\times\{(x,y)\}$, all parallel copies of the core $J=S^1\times\{(0,0)\}$, so it traces out a Legendrian isotopy between $J=S^1\times\{(0,0)\}$ and $J'=S^1\times\{(c,d)\}$. Since $J'\subset S^1\times\{(x,y):x^2+y^2\leq r'<r\}$, $J'$ and $K$ are unlinked, $J'$ co-bounds an annulus $A'$ with $S^1\times\{(1,0)\}$ and $K\cap A'=\emptyset$. This Legendrian isotopy extends to an ambient contact isotopy of $(S^1\times D^2,\xi_n)$ and can be arranged to be the identity on the boundary. 
\end{proof}

\begin{lemma}\label{unlinkonce}
The inverse of the isotopy traced out by $A$ is an ambient contact isotopy sending $J'$ to $J$ such that the image of $K$ is a Legendrian knot that crosses $J$ transversely to become unlinked from $J$.
\end{lemma}

Lemma \ref{unlinkonce} says that for homologous $K$ and $J$, Legendrian isotoping $K$ across $J$ does not change the value of $\widetilde{tb}_\Sigma(K,J)$. From the construction, such a Legendrian isotopy always exists. 

\begin{proposition}\label{unlink} 
Given a Legendrian knot $K$ homologous to $J=S^1\times\{(0,0)\}$ in $(S^1\times D^2,\xi_n)$, there exists an ambient contact isotopy, identity on the boundary, from $K$ to $K'$ such that $lk(K',J)=0$ and $\widetilde{tb}_n(K',J)=\widetilde{tb}_n(K,J)$.
\end{proposition}

\begin{lemma} 
After an unlinking Legendrian isotopy of $K$ as in Proposition \ref{unlink}, $tw_J(\xi_n,Fr_{\Sigma'})=n$ and $\widetilde{tb}_n(K)=tw_K(\xi_n,Fr_{\Sigma'})-n$, where $K'\cup J=\partial\Sigma'$.
\end{lemma}

\begin{proof} 
Once an unlinking Legendrian isotopy is applied to $K$, the Seifert surface $\Sigma'$ for $K'\cup J$ induces a Seifert framing $Fr_{\Sigma'}$ on $J$ which is equal to the product framing on $J$ (a push-off $J'$ of $J$ into $\Sigma'$ which defines the framing must vanish in the first homology of the complement). Therefore, $tw_J(\xi_n,Fr_{\Sigma'})=n$.
\end{proof}

Let $\kappa$ be a smooth knot type in $S^1\times D^2$, and let $\mathcal{L}_n(\kappa)$ denote \textit{the set of Legendrian representatives of $\kappa$ homologous to the core $J$ in $(S^1\times D^2,\xi_n)$}.

\begin{lemma} 
The function $\widetilde{tb}_n:\mathcal{L}_n(\kappa)\longrightarrow\Z$ is not bounded below.
\end{lemma}

\begin{proof} Take any Legendrian representative $K\in\kappa$ and stabilize $K$. The resulting knot $K'\in \kappa$ has $\widetilde{tb}_n(K')=\widetilde{tb}_n(K)-1$ with Seifert surface $\Sigma'$ obtained from $\Sigma$ by adding a half-disc (and smoothing corners). $\Sigma$ for $K$ and $\Sigma'$ are smoothly isotopic, and their boundaries are smoothly but not Legendrian isotopic,with $K'\in \kappa$. Repeating this lowers the relative Thurston-Bennequin invariant arbitrarily.
\end{proof}

\begin{theorem}\label{tbnbound}
The function $\widetilde{tb}_n:\mathcal{L}_n(\kappa)\rightarrow\Z$ is bounded above when $n<0$.
\end{theorem}

\begin{proof} Apply an unlinking Legendrian isotopy to a Legendrian representative $K$ of $\kappa$ in $(S^1\times D^2,\xi_n)$ homologous to the core $J=S^1\times\{0\}$. Consider a Legendrian unknot $K_0$ in $(S^3,\xi_{std})$ and a 2-disc $D$ bounded by $K_0$. Arrange that $tw_{K_0}(\xi_{std},Fr_D)=tb(K_0)=n<0$. Since $Tb=-1$ for the trivial knot type in $(S^3,\xi_{std})$, stabilization allows us to construct such a Legendrian unknot.  Now take a framed Legendrian neighborhood $(N(K_0),\xi_{std}\rvert_{N(K_0)})\simeq(S^1\times D^2,\xi_n)$, where $K_0$ is sent to $J=S^1\times\{0,0\}$, the Seifert framing on $K_0$ given by the product framing on $S^1\times\{0,0\}$. We have a framing-preserving contactomorphism $\varphi: (S^1\times D^2,\xi_n)\rightarrow(N(K_0),\xi_{std}\rvert_{N(K_0)})$. Then $\varphi(K)\subset N(K_0)$ is unlinked from $K_0$ and cobounds a surface $\varphi(\Sigma)$ with $K_0$ such that $tw_{K_0}(\xi_{std},Fr_{\varphi(\Sigma)})=n$. Since both framings $Fr_D$ and $Fr_{\varphi(\Sigma)}=\varphi_\ast(Fr_\Sigma)$ on $K_0$ are Seifert, they are both given by push-offs into each respective surface which vanish in $H_1(N(K_0)\setminus K_0)$. Thus, we can isotop the interiors of $D$ and $\varphi(\Sigma)$ so that in a neighborhood of $K_0$ they intersect only in $K_0$. Away from that neighborhood, however, the interior of $D$ may intersect the interior of $\varphi(\Sigma)$ and/or the knot $\varphi(K)\subset N(K_0)$. The possible intersections are arcs and closed curves, which are eliminated standardly (see \cite{georgi}) by locally isotoping the interior of $D$ without changing $tb_D(K_0)$ or $tb_{\varphi(\Sigma)}\varphi(K)$. Now $\varphi(\Sigma)$ and $D$ intersect only in $K_0$. So $\Sigma'=\varphi(\Sigma)\cup D$ is a Seifert surface for $\varphi(K)$. Since $\varphi$ is framing-preserving, $tw_K(\xi_n,Fr_\Sigma)=tw_{\varphi(K)}(\varphi_\ast(\xi_n),Fr_{\Sigma'})=tw_{\varphi(K)}(\xi_{std}\rvert_{N(K_0)},Fr_{\Sigma'})= tb_{\Sigma'}(\varphi(K))$. Therefore, $\widetilde{tb}_n(K)=tw_K(\xi_n,Fr_\Sigma)-n=tb_{\Sigma'}(\varphi(K))-n$ and $tb_{\Sigma'}(\varphi(K))$ is bounded above by the maximal Thurston-Bennequin invariant for the knot type of $\varphi(K)$ in $S^3$. This upper bound is independent of $\varphi$ and only depends on the smooth knot type of $K$ in $S^1\times D^2$. Therefore, $\widetilde{tb}_n(K)=tb_{\Sigma'}(\varphi(K))-n\leq Tb(\varphi(K))-n$ so $\widetilde{tb}_n(K)$ is bounded above. \end{proof}

For $K\subset(S^1\times D^2,\xi_n)$, let $\widetilde{Tb}_n(K)=\max\{\widetilde{tb}_n(K)\ \rvert\  K\text{ is of type } \kappa\}$.

\section{Limitations of the relative Legendrian knot invariants}

Let $K\subset(S^3,\xi_{std})$ be a null-homologous Legendrian knot, and let $J'=S^1\times \{0,\frac{1}{2}\}$ in $(S^1\times D^2,\xi_n)$ with $J\cup J'=\partial A$ for $A=S^1\times \{(x,y)\rvert x=0,0\leq y\leq\frac{1}{2}\}$. Then $\widetilde{tb}(J',J)=0$. Form the Legendrian connected sum $K'=J'\#K$ in $(S^1\times D^2,\xi_n)\#(S^3,\xi_{std})\cong(S^1\times D^2,\xi_n)$ with Seifert surface $\Sigma'$. Then $\widetilde{tb}_{\Sigma'}(K',J)=\widetilde{tb}_A(J',J)+tb_\Sigma(K)+1=tb_\Sigma(K)+1$ and $\widetilde{r}_{\Sigma'}(K',J)=\widetilde{r}_{A}(J',J)+r_\Sigma(K)=r_\Sigma(K)$.

Consider Chekanov's examples of Legendrian embeddings $K_1,K_2\subset(S^3,\xi_{std})$ of the $5_2$ knot with $tb(K_1)=tb(K_2)$ and $r(K_1)=r(K_2)$, yet not Legendrian isotopic (see \cite{chekanov}). The knots $K_1'=J'\#K_1$ and $K_2'=J'\#K_2$ in $(S^1\times D^2,\xi_n)$ are homologous to the core $J$ with $\widetilde{tb}_n(K_i',J)=tb(K_i)+1$ and $\widetilde{r}_n(K_i',J)=r(K_i)$.

\begin{lemma} 
The knots $K_1',K_2'\subset(S^1\times D^2,\xi_n)$ are not Legendrian isotopic.
\end{lemma}

\begin{proof} 
If $K_1'$ and $K_2'$ were Legendrian isotopic, then we embed $(S^1\times D^2,\xi_n)$ in $(S^3,\xi_{std})$ as a framed Legendrian neighborhood of an unknot $U$ with $tb=-n$ (see Section 4). The knots $K_i'$ are unlinked from $J$ so $\varphi(J)$ bounds a 2-disc $D$ disjoint from the Seifert surface of each $\varphi(K_i')$ . This produces two Legendrian isotopic knots with $tb(\varphi(K_i'))=\widetilde{tb}_n(K_i',J)-n=tb(K_i)-n+1$ and $r(\varphi(K_i'))=\widetilde{r}_n(K_i')=r(K_i)+r_D(U)$. Since $tb(K_1)=tb(K_2)$ and $r(K_1)=r(K_2)$, we have two Legendrian embeddings that are both stabilizations of of the Legendrian representatives $K_i$ of $5_2$ knot in $(S^3,\xi_{std})$ with equal invariants and are Legendrian isotopic, contradicting Chekanov's result. 
\end{proof}

This strongly suggests that the relative invariants exhibit the same limitations as their classical analogues. Arguing this in general would follow a similar argument.

\section{Legendrian knots which cobound an embedded annulus}

We will prove the following general theorem. The special case for $(S^1\times D_2,\xi_n)$ with $J\subset(S^1\times D_2,\xi_n)$ denoting the Legendrian core follows directly.

\begin{theorem}\label{boundannulus}
Let $K,J$ be Legendrian knots in a tight contact 3-manifold $(M,\xi)$ cobounding an embedded annulus $A\hookrightarrow (M,\xi)$ with $\widetilde{tb}_A(K,J)=0$ and $\widetilde{r}_A(K,J)=0$. There is a global contact isotopy of $(M,\xi)$ (fixing $\partial M$ if $\partial M\neq\emptyset$) sending $K$ to $J$.
\end{theorem}

\begin{lemma}
Given two Legendrian knots $K,J\subset (M,\xi)$ which cobound an annulus $A$ with $\widetilde{tb}_A(K,J)=0$, then $tw_K(\xi,Fr_A)=tw_J(\xi,Fr_A)\leq0$.
\end{lemma}

\begin{proof}
By Lemma \ref{reltb}, $tw_K(\xi,Fr_A)+tw_J(\xi,Fr_A)+\rvert r_A(K)+r_A(J)\rvert\leq-\chi(A)=0$ which implies $tw_K(\xi,Fr_A)=tw_J(\xi,Fr_A)\leq 0$.
\end{proof}

The above lemma and Honda's theorem (\cite{honda:one}), which extends Giroux's results to surfaces with boundary, imply that $A$ can be isotoped to be convex, rel $\partial A$, $C^0$-small near $\partial A$ and $C^\infty$-small away from $\partial A$. When $tw_K(\xi,Fr_A)=tw_J(\xi,Fr_A)<0$, we prove Theorem \ref{boundannulus} by foliating $A$ by Legendrian knots parallel to the boundary thus tracing out a Legendrian isotopy between $K$ and $J$. 

\begin{remark}
The argument do not apply when $tw_K(\xi,Fr_A)=tw_J(\xi,Fr_A)=0$. In this case, we stabilize the Legendrian knots and then apply this argument.  
\end{remark}

\begin{lemma}\label{trace}
Given two Legendrian knots $K,J\subset (M,\xi)$ which cobound an annulus $A$ with $\widetilde{tb}_A(K,J)=\widetilde{r}_A(K,J)=0$ and $tw_K(\xi,Fr_A)=tw_J(\xi,Fr_A)<0$, $A$ can be isotoped to be convex rel boundary with characteristic foliation with Legendrian leaves parallel to the boundary components.
\end{lemma} 

\begin{proof}[First Proof of Lemma \ref{trace}] 
Note $tw_K(\xi,Fr_A)=tw_J(\xi,Fr_A)<0$ implies that the dividing set $\Gamma_A$ has a nonempty intersection with $K$ and $J$. Take another convex annulus $A'$ with $\partial A'=K\cup J$ such that $A\cap A'=K\cup J$. Edge-Round along $K$ and $J$ to build a convex torus $T=A\cup A'$. This process uses standard framed Legendrian neighborhoods around $J$ and $K$ and replaces their intersection with $A\cup A'$ by a smooth embedded surface. Locally, this is a Legendrian isotopy of $K$ and $J$ to knots $K',J'$ (see \cite{honda:one}). By tightness and Giroux's Criterion, $\Gamma_T$ consists of an even number of parallel dividing curves that are not meridional. Note that $K, J\subset T$ are of the same homology class. Also, $tw_K(\xi,Fr_A)=tw_J(\xi,Fr_A)<0$ implies that $tw_{K'}(\xi,Fr_T)=tw_{J'}(\xi,Fr_T)<0$, so $J'$ and $K'$ intersect $\Gamma_T$. We can isotop $T$ to be foliated by leaves parallel to the knots $K'$ and $J'$ using Giroux's Flexibility theorem (\cite{giroux:convex}). This is a Legendrian isotopy of $K'$ and $J'$, sending them to $K''$ and $J''$ on the new convex torus $T'$ where they are Legendrian isotopic through the leaves. By the Legendrian Isotopy Extension theorem, the composition of these isotopies is a global contact isotopy. Note that we used $\widetilde{r}_A(K,J)=0$ to build $T$.
\end{proof}

\begin{proof}[Second Proof of Lemma \ref{trace}]  
Parametrize $A, K, J$ as $A=\varphi(S^1\times [-\frac{1}{2},\frac{1}{2}]), J=\varphi(S^1\times \{-\frac{1}{2}\}), K=\varphi(S^1\times \{\frac{1}{2}\})$ in $M$. Extend $\varphi$ to an embedding $S^1\times D^2\hookrightarrow (M,\xi)$. Consider a closed solid torus neighborhood $T$ of $A$. Consider a diffeomorphism $f:T\rightarrow S^1\times D^2$ such that $f:K\mapsto S^1\times \{p_1\}$ and $f:J\mapsto S^1\times \{p_2\}$ for $p_1,\ p_2\in D^2$. Let $S^1\times D^2$ be equipped with the contact structure $\xi_n$, where $tw_J(\xi,Fr_A)=tw_K(\xi,Fr_A)=-n$. Note $f(J)$ and $f(K)$ are Legendrian in $(S^1\times D^2,\xi_n)$, but they are also Legendrian in $(S^1\times D^2,f_\ast(\xi\rvert_T))$. So consider the Legendrian isotopy between them given by just sending $S^1\times \{p_1\}$ to $S^1\times \{p_2\}$ through parallel copies $g_t:S^1\times (1-t)p_1+tp_2$. This is a Legendrian isotopy inside $(S^1\times D^2,\xi_n)$. Since $f$ is a diffeomorphism rel boundary, it is a contactomorphism, and by the uniqueness of the tight contact structure $\xi_n$, there is an isotopy sending $\xi_n$ to $f_\ast(\xi\rvert_T)$. The inverse of this isotopy composed with the Legendrian isotopy from $f(K)$ to $f(J)$ and the inverse of $f$ yields a Legendrian isotopy from $J$ to $K$ in $T$.
\end{proof}

\begin{proof}[Third Proof of Lemma \ref{trace}] 
Note that $\Gamma_A$ has a component running from $K$ to $J$. To see this, assume $\Gamma_A$ consists only of boundary-parallel dividing arcs. Then in the construction of the convex torus $T=A\cup A'$ above, the dividing set would contain a trivial closed curve, contradicting Giroux's Criterion. Therefore, the annulus $A$ necessarily has a boundary-to-boundary dividing arc (an even number of these). 
\end{proof}

\begin{remark} 
Co-bounding an embedded annulus is a transitive relation of knots. In particular, the knots in this relation are smoothly isotopic.
\end{remark}

\begin{lemma}\label{threeannuli} 
Let $K_i$ be a framed knot with framing $Fr_i,\ i=1, 2, 3$ with $\partial A_i=K_i\cup K_{i+1},\ i=1,2$ for embedded annuli $A_i$. There exists an embedded annulus $A$ with $\partial A=K_1\cup K_3$ such that $tw_{K_1}(Fr_{A},Fr_1)=tw_{K_1}(Fr_{A_1},Fr_1)+K_1\cdot A_2$ and, similarly, $tw_{K_3}(Fr_A,Fr_3)=tw_{K_3}(Fr_{A_2},Fr_3)+K_3\cdot A_1$.
\end{lemma}

\begin{proof} 
Resolve the intersections of $A_1$ and $A_2$ to get an embedded annulus $A$. 
\end{proof}

\begin{lemma} 
Let $K_1$ and $K_2$ be Legendrian knots which cobound annuli $A_i$, respectively, with a Legendrian knot $J$ in a tight contact 3-manifold $(M,\xi)$. If $\widetilde{tb}_{A_1}(K_1,J)=\widetilde{tb}_{A_2}(K_2,J)$ and $\widetilde{r}_{A_1}(K_1,J)=\widetilde{r}_{A_2}(K_2,J)$, then $K_1$ and $K_2$ are Legendrian isotopic.
\end{lemma}

\begin{proof} 
Apply Lemma \ref{trace} together with Lemma \ref{threeannuli}. 
\end{proof}

\section{Legendrian knots isotopic to the core in $(S^1\times D^2,\xi_n)$.} 

Let $K$ be isotopic to $J$. The generator of $\ker\big{(}H_1(S^1\times D^2\setminus K)\to H_1(S^1\times D^2)\big{)}$ is a curve $\mu$ so that the $0$-framing of $K$ in $S^1\times D^2$ is defined by $K'$ with $[K']=0\cdot\mu$, where $[K']$ is unique up to a choice for a generator of the other factor in $H_1(\partial (S^1\times D^2\setminus K))\cong \Z\oplus\Z$. A Legendrian $K$ has a twisting number $tb_n(K)$ defined as $lk(K,K')$ for a push-off $K'$ in the normal direction to the contact planes along $K$, so $tb_n(K)=lk(K,K')=m$, where $m$ is the unique integer with $[K']=m\cdot\mu$. Now embed $(S^1\times D^2,\xi_n)$ in $(S^3,\xi_{std})$ as the framed neighborhood of a Legendrian unknot $U$ with $tb(U)=-n=tb_n(J)$. Then $K$ is a Legendrian unknot in $(S^3,\xi_{std})$ smoothly isotopic to $U$. Note that $tb_n(K)=tb_{D_K}(K)$, where $\partial D_K=K$ in $S^3$. Similarly, the global trivialization of $\xi_n$ in $S^1\times D^2$ given by $\partial/\partial z$ gives the rotation number $r_n(K)$. After embedding $(S^1\times D^2,\xi_n)$ standardly into $(S^3,\xi_{std})$, we have $r_{D_K}(K)=r_n(K)$.

With this in mind, we classify Legendrian knots smoothly isotopic to the Legendrian core in $(S^1\times D^2,\xi_n)$ with equal $tb_n$ and $r_n$.

\begin{lemma} 
Let $K$ be isotopic to $J=S^1\times \{(0,0)\}$ in $(S^1\times D^2,\xi_n)$, $K\cap J=\emptyset$. Then there exists an embedded annulus $A\hookrightarrow S^1\times D^2$ with $\partial A=K\cup J$.
\end{lemma}

\begin{proof} 
Embed $(S^1\times D^2,\xi_n)$ in $(S^3,\xi_{std})$ as the framed neighborhood of a Legendrian unknot $U$ with $tb(U)=tb_n(J)=-n$, The images of $K$ and $J$ are isotopic in $S^3$, and so are the discs $D_K$ and $D_J$ that they bound. In particular, we can isotop them outside the interior of $S^1\times D^2$ so that $D_K$ and $D_J$ coincide in there and in a curve $\gamma\subset\partial(S^1\times D^2)$. Consider the annuli $A_K=D_K\cap(S^1\times D^2)$ and $A_J=D_J\cap(S^1\times D^2)$ and resolve their intersections away from $\gamma$ as in \cite{georgi} to obtain an embedded annulus $A$ with $\partial A=K\cup J$ and the framings along $K$ and $J$ change by the same number $lk_{A_J}(K,J)=K\cdot A_J=lk_{A_K}(J,K)=J\cdot A_K$. 
\end{proof}

\begin{lemma}\label{annulusequal} 
For $A\hookrightarrow S^1\times D^2$ above, we have $\widetilde{tb}_A(K,J)=0$ and $\widetilde{r}_A(K,J)=0$.
\end{lemma}

\begin{proof}
For Legendrian $K$ and $J$, we have $tb_A(K)=tb_A(J)$ since $tb_n(K)=tb_n(J)$ so $\widetilde{tb}_A(K,J)=tb_A(K)-tb_A(J)=(tb_n(K)+lk_{A_J}(K,J))-(tb_n(J)+lk_{A_K}(J,K))=0$. Similarly, $\widetilde{r}_A(K,J)=r_n(K)-r_n(J)=r_A(K)-r_A(J)=0$. Therefore $A$ traces out a Legendrian isotopy between $K$ and $J$, by Theorem \ref{trace}.
\end{proof}

\begin{remark} 
Recall we assumed in \cite{georgi} that the Legendrian isotopy crossing the reference knot $J$ was locally embedded. Lemma \ref{annulusequal} shows this assumption is justified. The converse is not generally true, and finding an embedded annulus which traces out a Legendrian (or even smooth) isotopy is not generally possible. It is a good problem to find the obstructions for an isotopy to be embedded.
\end{remark}

\section{Further classification results}

Let $K, J_1,\dots, J_m$ be Legendrian knots in a contact 3-manifold $(M,\xi)$ with Seifert surface $\Sigma$. The relative invariants $\widetilde{tb}_\Sigma(K,J_1\cup\cdots\cup J_m)\:=tw_K(\xi,Fr_\Sigma)-\sum_{k=1}^m tw_{J_k}(\xi,Fr_\Sigma)$ and $\widetilde{r}_\Sigma(K,J_1\cup\cdots\cup J_m)\:=\omega(v_K)-\sum_{k=1}^m\omega(v_{J_i})$ are well-defined (\cite{georgi}), in particular, they are invariant under Legendrian isotopy of $K$ which fixes the $J_i$ (\cite{georgi}).

\begin{lemma} 
Let  $J_i=S^1\times \{p_i\}, p_i\in D^2$ be $m$ parallel copies of $J=S^1\times \{(0,0)\}$ in $(S^1\times D^2,\xi_n)$ and let $K$ be a Legendrian knot with $K\cup J_1\cup\cdots\cup J_m=\partial\Sigma$.  The $J_i$ may be Legendrian isotoped so that $tw_{J_i}(\xi_n,Fr_\Sigma)=tw_{J}(\xi_n,Fr_{S^1\times D^2})=tb_n(J)=-n$  and $\widetilde{r}_\Sigma(K,J_1\cup\cdots\cup J_m)=r_n(K)-\sum_{i=1}^mr_n(J_i)$.
\end{lemma}

\begin{proof} 
The method of unlinking Legendrian isotopy (Proposition \ref{unlink}) allows us to unlink $J_i$ from $K$ so that the Seifert framing and the product framing coincide, without changing the relative invariants (for the invariance of the relative rotation number, note that the contact structure is globally trivial). Then $tw_{J_i}(\xi_n,Fr_\Sigma)=tw_{J}(\xi_n,Fr_{S^1\times D^2})=-n$ so $\widetilde{tb}_\Sigma(K,J_1\cup\cdots\cup J_m)=tw_K(\xi_n,Fr_\Sigma)-mn$ and $\widetilde{r}_\Sigma(K,J_1\cup\cdots\cup J_m)=\omega(v_K)-\sum_{i=1}^m\omega(v_{J_i})=\omega(v_K)-m\omega(v_J)$.
\end{proof}

\begin{remark} 
Let $K$ be a Legendrian knot  which cobounds an $m$-punctured 2-disc $D$ with a collection of $m$ Legendrian copies of $J$ in $(S^1\times D^2,\xi_n)$. Then if $m$ is odd, $K$ is smoothly isotopic to $J$ and if $m$ is even then $K$ is homotopically trivial.
\end{remark}

\begin{lemma} 
Let $K_1$ and $K_2$ be two Legendrian knots each cobounding an $m$-punctured 2-disc $D_i$ with $m$ copies of $J$ in $(S^1\times D^2,\xi_n)$. Assume $\widetilde{tb}_{D_1}(K_1,J_1'\cup\cdots\cup J_m')=\widetilde{tb}_{D_2}(K_2,J_1''\cup\cdots\cup J_m'')$ and $\widetilde{r}_{D_1}(K_1,J_1'\cup\cdots\cup J_m')=\widetilde{r}_{D_2}(K_2,J_1''\cup\cdots\cup J_m'')$. Then if $m$ is odd, there exists an embedded annulus $A$ with $\partial A=K_1\cup K_2$ such that $\widetilde{tb}_A(K_1,K_2)=0$ and $\widetilde{r}_A(K_1,K_2)=0$. If $m$ is even, there exists an embedded 2-dsic $D_i$ with $K_i=\partial D_i$ such that $tb_{D_1}(K_1)=tb_{D_2}(K_2)$ and $r_{D_1}(K_1)=r_{D_2}(K_2)$.
\end{lemma}

\begin{proof} 
We use Proposition \ref{unlink} to isotop all $J_i'$ and $J_j''$ to a neighborhood $N(\widetilde{J})$ of a copy of $J$ with the $K_i\subset(S^1\times D^2)\setminus N(\widetilde{J})$. The relative invariants are fixed and $tw_{J_i'}(\xi_n,Fr_{D_1})=tw_{J_j''}(\xi_n,Fr_{D_2})=tb_n(J)=-n$ and $\widetilde{r}_{D_1}(J_i')=r_n(J_i')=\widetilde{r}_{D_2}(J_j'')$. Now Legendrian isotop each $J_i'$ to a $J_j''$ through parallel copies of $J$. This may create circle intersections between $D_1$ and $D_2$, but there is an arrangement of $J_i'$ getting mapped to the $J_j''$ which avoids all circle intersections. Also, $tw_{J_i'}(\xi_n,Fr_{D_1})=-n=tw_{J_j''}(\xi_n,Fr_{D_2})$ implies that $tw_{K_1}(\xi_n,Fr_{D_1})=tw_{K_2}(\xi_n,Fr_{D_2})$ and similarly $r_n(K_1)=\omega(v_{K_1})=r_n(K_2)=\omega(v_{K_2})$. After all circles are eliminated, $\Sigma=D_1\cup D_2$ has no self-intersections near the $J_i'=J_i''$ and arc self-intersections $\alpha_k$ running only from $K_1$ to $K_2$ and possibly some other circle intersections. If $m$ is odd, then $K_1$ and $K_2$ are smoothly isotopic to the core in $S^1\times D^2$ and are thus smoothly isotopic, with the union of the $K_i$ and $\alpha_k$ cobounds a collection of disjoint 2-discs in $S^1\times D^2$, whose union is an annulus $A$ with $K_1\cup K_2=\partial A$. If $m$ is even, then $K_1$ and $K_2$ are (trivial and therefore) isotopic in $S^1\times D^2$, and the above annulus traces out the Legendrian isotopy between them.
\end{proof}

\begin{theorem} 
Let $K_1$ and $K_2$ be two smoothly isotopic Legendrian knots each cobounding an $m$-punctured 2-disc $D_i$ with $m$ copies of $J$ in $(S^1\times D^2,\xi_n)$. If $\widetilde{tb}_{D_1}(K_1,J_1'\cup\cdots\cup J_m')=\widetilde{tb}_{D_2}(K_2,J_1''\cup\cdots\cup J_m'')$ and $\widetilde{r}_{D_1}(K_1,J_1'\cup\cdots\cup J_m')=\widetilde{r}_{D_2}(K_2,J_1''\cup\cdots\cup J_m'')$, then $K_1$ and $K_2$ are Legendrian isotopic.
\end{theorem}

This can be translated to links in $(S^3,\xi_{std})$ by considering $J$ as a trivial link component. Consider two smoothly isotopic $(m+1)-component$ Legendrian links $L_1=K_1\cup J_1\cup J_2\cup\cdots\cup J_m$ and $L_2=K_2\cup J_1'\cup J_2'\cup\cdots\cup J_m'$ in $(S^3,\xi_{std})$ consisting only of trivial components and with Seifert surfaces $D_1,D_2$ which are $m$-punctured 2-discs such that $K_i$ bounds a 2-disc whose interior is disjoint from $D_i$ for $i=1,2$. Now take a Legendrian unknot $U_i$ that links with $L_i$ so that $lk_{D_1}(U_1,J_k)=lk_{D_2}(U_2,J_k')=\pm1$ so the new links $L_i\cup U_i$ are smoothly isotopic. Then after appropriate stabilizations of some or all of their components, the links $L_1\cup U_1$ and $L_2\cup U_2$ are Legendrian isotopic. To see this, take a framed Legendrian neighborhood of $U_i$, such that the $J_i$ are meridians, isotop the $U_i$ to coincide, and extend to a global contact isotopy. Then the complement of the now single solid torus is a solid torus to which the above theorem applies.

\section{Connected sums and Legendrian simple knots.}

Etnyre and Honda \cite{etnyrehonda:knots2} connected sums of Legendrian simple knots extensively. We are interested in the relative version of their results. Let $\mathfrak{L}$ be the set of knot types whose Legendrian embeddings in $(M_1,\xi_1)$ are Legendrian simple (i.e., classified by their classical invariants). Let $\mathfrak{L}'$ be the set of knot types $K_1\# J$, where $K_1\in\mathfrak{L}$ and $J$ is a Legendrian knot in a tight contact 3-manifold $(M_2,\xi_2)$. Consider an embedded annulus $A$ in $M_2$ bounded by $J$ and another Legendrian knot $K_2$, so $K_1\# K_2$ in $(M_1, \xi_1)\#(M_2,\xi_2)$ is homologous to $J$. For a Legendrian representative $K$ of $K_1\# K_2$ in $(M,\xi)$, the relative invariants $\widetilde{tb}(K, J)$ and $\widetilde{r}(K, J)$ are well-defined (??? and Theorem \ref{colin}). Let $\mathfrak{L}''\subset \mathfrak{L}'$ denote knot types in $\mathfrak{L}'$ whose Legendrian representatives in $(M,\xi)$ are classified by the relative invariants with respect to $J$ (call them \textit{relatively Legendrian simple with respect to $J$}).

\begin{question} 
When does the relative connected sum give a bijection $\mathfrak{L}\iff\mathfrak{L}''$? What are the obstructions, and when is the map only injective or surjective?
\end{question}

Etnyre-Honda \cite{etnyrehonda:knots2} constructed connected sums of Legendrian simple knots in $S^3$ which are not Legendrian simple so $\mathfrak{L}\rightarrow\mathfrak{L}''$ is not always a bijection. 

\begin{lemma} 
There is a bijective correspondence between $\mathfrak{L}$ and $\mathfrak{L}''$ in the case when $(M_i,\xi_i)=(S^3,\xi_{std}),\ i=1,2$ and $J\subset(M_2,\xi_2)$ is the unknot.
\end{lemma} 

This lemma is fairly straightforward to prove. We prove a generalization below.

\begin{lemma} \label{ministr}
There is a bijective correspondence between $\mathfrak{L}$ and $\mathfrak{L}''$ for $(M_1,\xi_1)=(S^3,\xi_{std}),\ (M_2,\xi_2)=(S^1\times D^2,\xi_n)$, and $J\subset(S^1\times D^2,\xi_n)$ is the Legendrian core.
\end{lemma} 

\begin{proof} The mapping is given by the relative Legendrian connected sum. To see that it is onto, we take a knot type $\kappa_{S^1\times D^2}\in\mathfrak{L}''$ and note that the splitting of a Legendrian connected sum gives a unique knot type $\kappa_{S^3}$, so we need to show $\kappa_{S^3}\in\mathfrak{L}$. Given Legendrian knots $K_i'$ in $(S^3,\xi_{std})$ with $K_i'=\partial\Sigma_i'$ of smooth knot type $\kappa_{S^3}$ with equal classical invariants $tb_{\Sigma_i'}(K_i')$ and $r_{\Sigma_i'}(K)$, form the Legendrian connected sums $K_i=K_i'\#J_i'\subset (S^1\times D^2,\xi_n)\cong(S^3,\xi_{std})\#(S^1\times D^2,\xi_n)$. Note that $J_i'$ has equal relative invariants with respect to $J$, and $K_i'=K\# J'$ is homologous to $J$ via the Seifert surface $\Sigma'=\Sigma\# A$. Therefore, $\widetilde{tb}_{\Sigma_i'}(K_i',J)=tb_{\Sigma_i}(K_i)+1$ and $\widetilde{r}_{\Sigma_i'}(K_i',J_i')=r_{\Sigma_i}(K_i)$. So the $K_i'$ are relatively Legendrian isotopic in $(S^1\times D^2,\xi_n)$. Extend this Legendrian isotopy to a contact isotopy of $(S^1\times D^2,\xi_n)$ such that the separating 2-spheres and the 3-balls they bound in $(S^1\times D^2,\xi_n)$ coincide. Thus, the Legendrian isotopy reduces to an isotopy within the (now single) 3-ball containing the parts of the $K_i'$ corresponding to the summands $K_i$. Now consider a contact isotopy of $(S^3,\xi_{std})$ which Legendrian isotops the knots $K_i'$ so that they coincide at a point (on a neighborhood of that point, in fact). Such an isotopy exists (in the front projection, it is a horizontal and vertical translation). So then we can use the Legendrian isotopy fixing a common point that we found in the 3-ball above. The composition of these two Legendrian isotopies gives us the desired Legendrian isotopy from $K_1$ onto $K_2$ in $(S^3,\xi_{std})$. Note that we are using the classification of tight contact structures in $S^3$ and the standard 3-ball. To see that the mapping is injective, note that the relative Legendrian connect sum defines the knot type $\kappa_{S^1\times D^2}$ uniquely for a given knot type $\kappa_{S^3}$ in the smooth category. For a Legendrian simple $\kappa_{S^3}$, we show that $\kappa_{S^1\times D^2}$ is relatively Legendrian simple. Consider two Legendrian knots $K_1$ and $K_2$ in $(S^1\times D^2,\xi_n)$ of knot type$\kappa_{S^1\times D^2}$ with equal relative invariants. Then $K_i=K_i'\#J_i$ for $K_1'$ and $K_2'$ of knot type $\kappa_{S^3}$. The relative Legendrian connected sum replaces a standard 3-ball neighborhood of a point $p_i\in J_i'$ with the standard 3-ball complement of a 3-ball neighborhood of a point on each knot $K_i'$. Use a contact isotopy to Legendrian isotop $J_1'$ to $J_2'$ so that the points $p_i$ coincide and the 3-ball neighborhoods of those coincide as well. After the connected sum, we further extend the isotopy using the Legendrian isotopy between the $K_i'$ from a 3-ball complement of a point in $S^3$ (this is contained in the standard 3-ball that they are embedded in). Thus, $K_i=K_i'\#J_i'$ are Legendrian isotopic in ${S^1\times D^2}$ provided that the relative classical invariants of the knots $K_i=K_i'\#J_i'$ with respect to $J$ are equal. 
\end{proof}

\begin{remark}
In order to piece together the Legendrian isotopies, equality of the relative invariants is not sufficient, we may need to distribute stabilizations among the components. Thus the bijective correspondence holds up to an equivalence relation that accounts for this (see Theorem \ref{relstr} and compare with Theorem 3.4 in \cite{etnyrehonda:knots2}).
\end{remark}

\begin{remark} 
The results in Section 7 follow from Lemma \ref{ministr}. Moreover, the argument generalizes to classify all Legendrian simple knot types in $(S^3,\xi_{std})$ as relatively Legendrian simple knot types in $(S^1\times D^2,\xi_n)$.
\end{remark} 

Lemma \ref{ministr} applies to Legendrian links in $(S^3,\xi_{std})$ with a trivial component.

\begin{theorem} 
Let $K\in \mathfrak{L}$ and $U$ be an unknot. Legendrian simple links $K\cup U$ in $(S^3,\xi_{std})$ are classified by the link type and relative invariants of $K$ relative to $U$.
\end{theorem}

\section{The relative structure theorem}

A homologous knot pair $(K,J)$ in a tight contact 3-manifold $(M,\xi)$ is \textit{relatively prime} if $(K,J)=(K_1,J_1)\#(K_2,J_2)$ in $(M_1,\xi_1)\#(M_2,\xi_2)$ implies $K_1\subset M_1$ or $K_2\subset M_2$ is the unknot. Let $S_\pm(K)$ be the positive/negative stabilization of $K$ and $\mathfrak{L}_{(\kappa, M,\xi)}$ be the set of isotopy classes of Legendrian representatives of $\kappa$ in $(M,\xi)$. We have a relative version of Theorem 3.4 by Etnyre-Honda in \cite{etnyrehonda:knots2}.

\begin{theorem}\label{relstr} 
Let $(\kappa,\gamma)=(\kappa_1,\gamma_1)\#\cdots\#(\kappa_n,\gamma_n)$ be a relative connected sum decomposition in a tight $(M,\xi)$ into relatively prime pairs $(\kappa_i,\gamma_i)\subset(M_i,\xi_i)$. The map $\Bigg{(}\dfrac{\big{(}\mathfrak{L}_{\kappa_1},\mathfrak{L}_{\gamma_1}\big{)}\times\cdots\times \big{(}\mathfrak{L}_{\kappa_n},\mathfrak{L}_{\gamma_n}\big{)}}{\sim}\Bigg{)}\longrightarrow \Big{(}\mathfrak{L}_{\kappa_1\#\cdots\#\kappa_n},\mathfrak{L}_{\gamma_1\#\cdots\#\gamma_n}\Big{)}$ is a bijection, where $\sim$ is defined by

\noindent(1) $\Big{(}\dots,(S_\pm(K_i),J_i),\dots,(K_j,S_\pm(J_j)),\dots\Big{)}\sim$\\

$\sim\hspace{2.2in}\Big{(}\dots,(K_i,S_\pm(J_i)),\dots,(S_\pm(K_j),J_j),\dots\Big{)}$\\
\noindent (2) $\Big{(}(K_1,J_1),\dots,(K_n,J_n)\Big{)}\sim\sigma\Big{(}(K_1,J_1),\dots,(K_n,J_n)\Big{)}$, where $\sigma$ is a permutation of $(\kappa_i,\gamma_i)$ so that $\sigma(M_i,\xi_i)$ is isotopic to $(M_i,\xi_i)$ and $\sigma(\kappa_i,\gamma_i)=(\kappa_i,\gamma_i)$ for all $i$.
\end{theorem}

This result follows directly from the structure theorem in \cite{etnyrehonda:knots2}, together with the construction of the relative Legendrian connected sum and the well-definedness of the relative invariants in \cite{georgi}.

\end{document}